\documentclass[12pt,amstex]{amsart}
\usepackage{}
\usepackage{amsfonts}
\usepackage{amsfonts}

\usepackage{epsfig}
\usepackage{amsmath}
\usepackage{amssymb}
\usepackage{amscd}
\usepackage{graphicx}
\usepackage{color}

\usepackage{pstricks}

\usepackage{tocvsec2}

\input xy
\xyoption{all}

\newtheorem{thm}{Theorem}[section]

\newtheorem{prop}[thm]{Proposition}

\theoremstyle{definition}

\theoremstyle{remark}

\numberwithin{equation}{section}

\def \B {\mathcal{B}}
\def \F {\mathcal{F}}

\begin{document}

\title[Auslander-Yorke¡¯s type
dichotomy theorems]{Auslander-Yorke¡¯s type
dichotomy theorems for stronger version $r$-sensitivity}

\author[K. Liu]{Kairan Liu}
\address{K. Liu: Department of Mathematics, University of Science and Technology of China,
Hefei, Anhui, 230026, P.R. China}
\email{lkr111@mail.ustc.edu.cn}
\author[X. Zhou]{Xiaomin Zhou}
\address{X. Zhou: Department of Mathematics, Huazhong University of Science and Technology,
Wuhan, Hubei, 430074, P.R. China}
\email{zxm12@mail.ustc.edu.cn}

\date{}

\begin{abstract}
In this paper, for $r\in \mathbb{N}$ with $r\geq 2$ we consider several stronger version $r$-sensitivities and measure-theoretical $r$-sensitivities by analysing subsets of nonnegative integers, for which the $r$-sensitivity occurs. We obtain an Auslander-Yorke's type dichotomy theorem: a minimal topological dynamical system is either thickly $r$-sensitive or an almost $m$ to one extension  of its maximal equicontinuous factor for some $m\in \{1,\cdots, r-1\}$.
\end{abstract}
\maketitle

\section{Introduction}
In the paper, sets of all integers, nonnegative integers and
natural numbers are denoted by $\mathbb{Z},\ \mathbb{Z}_+$ and $\mathbb{N}$ respectively.
Throughout this paper by a {\em topological dynamical system }(TDS for short) we mean a pair $(X, T)$,
where $X$ is a compact metric space with a metric $d$ and $T : X\rightarrow X$ is homeomorphism. For a TDS $(X,T)$ there exist invariant Borel probability measures. Given a $T$-invariant Borel probability measure $\mu$, we denote the induced {\em measurable  dynamical system}
$(X,\mathcal{B}_X,\mu,T)$, where $\mathcal{B}_X $ is the Borel $\sigma$-algbra of $X$. We also let $\mathcal{B}_{X,\mu}^+=\{B \in \mathcal{B}_X | \mu (B)>0 \}\ $.

A TDS $(X, T)$ is called {\em equicontinuous} if $\{T^n :n\ge 0\}$ is equicontinuous at any point of $X$. Each dynamical system admits a maximal equicontinuous factor. In fact, this factor is related to the regionally proximal relation of the system. Let $(X, T)$ be a TDS. {\em The regionally proximal relation} $Q(X,T)$
 of $(X, T)$ is defined as: $(x ,y )\in Q(X,T)$
 if and only if for any $\epsilon>0$
 there exist $x',y'\in X$  and $m\in \mathbb{Z}_+$
 such that $d(x,x')<\epsilon$, $d(y,y')<\epsilon$ and {$d(T^mx',T^m y')<\epsilon$}.
 Observe that $Q(X,T)\subset X\times X$
 is closed and positively invariant (in the sense that if $(x ,y)\in Q(X,T)$
  then $T\times T(x ,y)\in Q(X,T)$), which induces the maximal equicontinuous factor $(X_{eq} ,S)$
 of $(X, T)$. And if $(X, T)$ is minimal, i.e., $\text{orb}(x,T):=\{T^n x:n\in \mathbb{Z}\}$
 is dense in $X$ for any $x\in X$, then $Q(X,T)$ is in fact an equivalence relation
 by \cite{Au,BHM2000,EG1960,Veech1968} and \cite[Proposition A.4]{HY}. Denote by $\pi_{eq}:(X,T)\rightarrow(X_{eq},S)$
the corresponding factor map.

In \cite{R} Ruelle introduced the notion of sensitivity (sensitive dependence on initial conditions), which is the opposite to the notion of equicontinuity. According to the works by Guckenheimer \cite{Gu}, Auslander and Yorke \cite{AY},  a TDS $(X,T)$ is called {\em sensitive} if there exists $\delta >0$ such that
 in any non-empty open subset $U$ of $X$ there are $x,y\in U$
 and  $n\in \mathbb{N}$ with $d(T^n x,T^n y)>\delta$. Auslander and Yorke \cite{AY} proved the following dichotomy theorem:
a minimal system is either equicontinuous or sensitive (see also \cite{GW}).

The notion of sensitivity was generalized by measuring the set of nonnegative integers for which the sensitivity occurs \cite{HKKZ,HKZ,L,LY, M,YY}. For a subset $A$ of natural numbers $\mathbb{N}$,  we say $A$ is
\begin{enumerate}
\item {\em thick} if for any $ k \in \mathbb{N}$ we can find some $n \in \mathbb{N}$ such that $ \{ n, n+1,\cdots,n+k \}  \subset A$;
\item {\em syndetic} if there exists some $k \in \mathbb{N}$ such that for every $n \in \mathbb{N}$ we have $ \{ n,n+1,\cdots,n+k \} \cap A \neq \emptyset $;
\item {\em thickly syndetic} if $\{ n\in \mathbb{Z_+}: \{ n,n+1,\cdots,n+k \} \subset A \} $ is syndetic for each $k\in \mathbb{N}$.
\end{enumerate}
Thick sensitivity, thickly syndetical sensitivity and multi sensitivity were introduced and investigated in \cite{M,LLW}.
Huang, Kolyada and Zhang \cite[Theorem 3.1]{HKZ} showed that
a minimal system is either thickly sensitive or an almost one-to-one extension of its maximal equicontinuous factor.
Ye and Yu introduced block (resp. strongly) thickly (resp. IP) sensitivity and proved several Auslander-Yorke's type dichotomy theorems in \cite{YY}.

There are also several authors to study the measure-theoretic sensitivity
 \cite{ABC,ADFKLS,CJ,G,HMS,HLY,JKL,L,LY,WW2009,Y}. Huang, Lu and Ye \cite{HLY} introduced a notion called
sensitivity for an invariant Borel probability measure $\mu$ and proved that a minimal system is either equicontinuous or sensitive for $\mu$.
Wu and Wang \cite{WW2009} introduced $\F$-$\mu$-pairwise sensitivity and $\F$-$\mu$-sensitivity, where $\F$ is a family, and
investigated  when the two notions coincide. Yu \cite{Y} discussed $\F$-sensitivity for $\mu$ and proved the equivalence between some different measure sensitivity for a minimal TDS.

Recently Zou considered stronger version $r$-sensitivities and also discussed equivalence between some strong version sensitivities for transitive or minimal TDS in \cite{Z}, where $r$-sensitivity was firstly  introduced by Xiong in \cite{X} which is a stronger version sensitivity (see also \cite{SYZ,YZ}).

Inspired by the previous works, our aim in this paper is to investigate stronger version $r$-sensitivities, measure-theoretical $r$-sensitivities,  and give an  Auslander-Yorke's type dichotomy theorem for stronger version $r$-sensitivity. More precisely, for a TDS $(X,T)$ with a $T$-invariant Borel probability measure $\mu$, $\delta>0$ and $r\in \mathbb{N}$ with $r\geq 2$, and a non-empty subset $U$ of $X$, we put
$$ N_T(U,\delta;r):=\{n\in \mathbb{N}:\exists x_1,x_2,\cdots,x_r\in U \text{ such that} \min \limits_{1\leq i\neq j \leq r}d(T^nx_i,T^nx_j)>\delta \} .$$
It is easy to see that $(X,T)$ is sensitive if and only if there exists $\delta>0$ such that $N_T (U, \delta;2)$
is infinite for every non-empty open subset $U$ of $X$. Moreover we say $(X,T)$ is
 \begin{enumerate}
 \item {\em thickly $r$-sensitive  (resp. for $\mu ) $} if there exists $\delta>0$ such that $N_T(U,\delta;r)$ is thick for any non-empty open subset $U$ of $X$ (resp. for any set $U\in \mathcal{B}_{X,\mu}^+$);
\item  {\em thickly syndetically $r$-sensitive (resp. for $\mu$)} if there exists $\delta>0$ such that $N_T(U,\delta;r)$ is thickly syndetic for any non-empty open subset $U$ of $X$ (resp. for any set $U\in \mathcal{B}_{X,\mu}^+$);
\item  {\em multi-$r$-sensitive (resp. for $\mu$}) if there exists $\delta>0$ such that $$\bigcap \limits_{i=1}^{k} N_T(U_i,\delta;r)\neq \emptyset $$ for any $k\in \mathbb{N}$ and any non-empty open subsets $U_1,U_2,\cdots, U_k$ of $X$ (resp. for any finite collection $U_i\in \mathcal{B}_{X,\mu}^+,\ i=1,\cdots,k$).
\end{enumerate}

The main results of this paper are follows:
\begin{thm} \label{thm-1} Let $(X,T)$ be a minimal TDS,  $r\in \mathbb{N}$ with $r\geq 2$, and $ \pi_{eq} :(X,T)\rightarrow (X_{eq},S)$ be the maximal equicontinuous  factor of $(X,T)$. Then $(X,T)$ is either thickly $r$-sensitive or an almost $m$ to one extension  of its maximal equicontinuous factor (i.e., there exists a dense $G_\delta $ subset $A$ of $X_{eq}$ such that $ \# \pi_{eq}^{-1}(y)=m $ for any $ y\in A $) for some $m\in \{1,2,\cdots, r-1\}$.
 \end{thm}
 \begin{thm} \label{thm-2} Let $(X,T)$ be a minimal TDS with a $T$-invariant Borel probability measure $\mu$,  $r\in \mathbb{N}$ with $r\geq 2$, and $ \pi_{eq} :(X,T)\rightarrow (X_{eq},S)$ be the maximal equicontinuous  factor of $(X,T)$. Then the following statements are equivalent:
\begin{enumerate}
  \item[1).] $(X,T)$ is multi-$r$-sensitive for $\mu$.
  \item[2).] $(X,T)$ is thickly $r$-sensitive for $\mu$.
  \item[3).] $(X,T)$ is thickly syndetically  $r$-sensitive for $\mu$.
  \item[4).] For any $ m\in \{1,2,\cdots,r-1\} $, $\pi_{eq}$ is not almost $m$ to one  extension.
\end{enumerate}
\end{thm}

The paper is organized as follows. In Section 2, we firstly recall some definitions and some
related lemmas, and then we prove Theorem \ref{thm-1}. In Section 3, we study some related measure-theoretic $r$-sensitivity and prove Theorem \ref{thm-2}.

\section{Proof of Theorem \ref{thm-1}}

In this section we are to prove Theorem \ref{thm-1}. For that we need some notation and Propositions.
Let $(X,T)$ be a minimal TDS, and $ \pi_{eq} :(X,T)\rightarrow (X_{eq},S) $ be the maximal equicontinuous factor of $(X,T)$. For $y\in X_{eq}$ we define:
$$r_{eq}(y):=\sup\{k\in \mathbb{N}:\exists x_1,...,x_k\in \pi_{eq}^{-1}(y) \text{ s.t. } (x_i,x_j) \text{ is  distal},\ \forall 1\leq i\neq j\leq k \},\ $$
where a pair $(x,y)\in X\times X $ is called {\em distal}, if $\inf_{n\in \mathbb{Z}}d(T^nx,T^ny) >0 $.
According to Lemma 2.2 in \cite{Z}, $r_{eq}(y)$ is a constant function on $X_{eq}$, which is denoted as $r_{eq}(X,T)$.
\begin{prop}\label{pro-0}\cite[Proposition 2.6]{Z}
Let $(X,T)$ be a minimal TDS and $\pi_{eq} :(X,T)\rightarrow (X_{eq},S)$ be the  maximal  equicontinuous factor  of $(X,T)$. If $ \# \pi^{-1}_{eq}(y_0)<\infty $ for some $y_0 \in X_{eq}$, then
$$r_{eq}(X,T)=\min_{y\in X_{eq}}\#\pi_{eq}^{-1}(y)\in [1,+\infty)$$
and  $\pi_{\text{eq}}$ is  almost $r_{eq}(X,T)$ to one extension, that is,
$$\#\pi_{eq}^{-1}(y) = r_{eq}(X,T)$$
holds for $y$ in a dense
$G_\delta$ subset of $X_{eq}$.
\end{prop}

\begin{prop}\label{pro-1} Let $(X,T)$ be a minimal TDS, $r\in \mathbb{N}$ with $r\geq 2 $, and $\pi_{eq} :(X,T)\rightarrow (X_{eq},S)$ be the  maximal  equicontinuous factor  of $(X,T)$.
If we put
$$ \phi_r(y):=\sup \{\min \limits_{1\leq i< j\leq r} d(x_i,x_j) : x_{1},x_{2},\cdots,x_{r} \in \pi^{-1}(y)\}$$
for $ y\in X_{eq}$ and let $$\eta_{r}=\inf_{y\in X_{eq}} \phi_r(y),$$
then $\eta_{r}=0$ if and only if  there exists $ m\in\{1,2,\cdots,r-1\}$ such that $ \pi_{eq} $ is almost $m$ to one extension.
\end{prop}

\begin{proof} We partly follow the arguments in the proof of Proposition 4.4 in \cite{Z}. If $\eta_{r}=0$ then there exist $\{y_i\}_{i=1}^\infty\subset X_{eq}$ such that $\lim_{i\rightarrow +\infty} \phi_r(y_i)=0$. Note that $\phi_r$  is an upper semi-continuous function from $X_{eq}$ to $[0,+\infty)$, we may take a continuous point $y_0\in X_{eq}$ of $\phi_r$.
For any fixed $\varepsilon >0 $, we choose an open neighborhood $V_{\varepsilon}$ of $y_0$ such that
$ |\phi_r(y)-\phi_r(y_0)|< \varepsilon$ for  $y\in V_{\varepsilon}$.
Since $(X,T)$ is minimal and $(X_{eq},S)$ is also minimal, there exists $\ell_k\in \mathbb{N}$ with $\bigcup_{j=0}^{\ell_k}S^{-j}V_{\varepsilon}= X_{eq}$. Then we can find
$t\in \{0,1,\cdots,\ell_k\}$ such that
$E_t^{\varepsilon}:=\{i\in \mathbb{N}:S^ty_i\in V_{\varepsilon}\}$ is an infinite set. Since $\pi_{eq}^{-1}(S^t y_i) = T^t \pi_{eq}^{-1}(y_i)$, so as $ i\in E_t^{\varepsilon}\rightarrow +\infty $
$$ \phi_r(S^ty_i)=\sup \{\min \limits_{1\leq h< j\leq r} d(T^tx_h,T^tx_j) : x_{1},x_{2},\cdots,x_{r} \in \pi^{-1}(y_i)\} \rightarrow 0 .$$
Then $|\phi_r(y_0)|\le \varepsilon$. This implies $\phi_r(y_0)=0 $ as $\varepsilon $ is arbitrary. Moreover we have $\#\ \pi_{eq}^{-1}(y_0)\leq r-1$ from the definition of $\phi_r$.
Let $m=r_{eq}(X,T)$. Then $m\in \{1,2,\cdots,r-1\}$. By Proposition \ref{pro-0}, $\pi_{eq}$ is almost $m$ to one extension. Finally the other hand is obviously true.
\end{proof}
  Let $(X,T)$ be a TDS with a $T$-invariant Borel probability measure $\mu$, and $n\ge 2$. Then
$(x_i)_1^n\in X^n$ is a {\it sensitive $n$-tuple} for $\mu$,
if $(x_i)_1^n$ is not on the diagonal $$\Delta_n(X)=\{ (x,...,x)\in X^{n}: x \in X \},$$ and for any
open neighborhood $U_i$ of $x_i$ and any $A\in \B_{X,\mu}^+$ there is
$k\ge 0$ such that $A\cap T^{-k}U_i\neq \emptyset$ for
$i=1,2,\cdots,n$.
Denote by $S_n^\mu(X,T)$ the set of all sensitive $n$-tuples for
$\mu$.

For $n\ge 2$
the $n$-{\it regionally proximal relation} is defined as
\begin{align*}
&Q_{n}(X,T)= \{ (x_i)_{i=1}^n\in X^n: \text{ for any $\epsilon>0$
there exist $x_1',\cdots,x_n'\in X$ and $m\in \mathbb{Z}_+$ }\\
&\text{ with } d(x_i,x_i')<\epsilon \text{ for all }1\le i\le n   \text{ and }
d(T^m x_i',T^m x_j')<\epsilon \text{ for all } 1\le i\le j\le n \}.
\end{align*}
For a minimal TDS $(X,T)$, $Q_{2}(X,T)$ is a closed invariant equivalence relation which induces the maximal equicontinuous factor of $(X,T)$.

We have the following proposition
for a minimal TDS.
\begin{prop}\label{xyz}\cite[Proposition 6.8 and Corollary 6.9]{HLY}
If $(X,T)$ is a minimal  TDS with a $T$-invariant Borel probability measure $\mu$ and $n\in \mathbb{N}$ with $n\geq 2 $,  then
 \begin{enumerate}
 \item if $(x_i ,x_{i+1} )\in Q(X,T)$  for $i=1,\cdots,n-1$, then $(x_1,\cdots,x_n)\in Q_{n}(X,T)$.

 \item $S_n^\mu(X,T)=Q_{n}(X,T)\setminus\Delta_n(X)$.
\end{enumerate}
\end{prop}

\begin{prop}\label{pro-2} Let $(X,T)$ be a minimal TDS, $r\in \mathbb{N}$ with $r\geq 2 $, and $ \pi_{eq} :(X,T)\rightarrow (X_{eq},S)$ be the maximal equicontinuous  factor of $(X,T)$. Then $\eta_{r}>0$ if and only if $(X,T)$ is $thickly \ r$-$sensitive $.
\end{prop}
\begin{proof} Assume that $(X,T)$ is  thickly $r$-sensitive with a sensitive constant $\delta>0$. We are to show that $\eta_{r}>0$. If it is not true, there exists $ m\leq r-1 $ such that $ \pi_{eq} $ is almost $m$-to-one extension by Proposition \ref{pro-1}. Actually there exists $y_0\in X_{eq}$ such that $\#\ \pi_{eq}^{-1}(y_0)=m\leq r-1.$

Let $\pi^{-1}_{eq}(y_0)=\{x_1, \cdots, x_m\}$.
We take $$W=\{x\in X: d(x,x_i)<\frac{\delta}{3} \text{ for some }i\in \{1,2,\cdots,m\}\}.$$
Then $W$ is open and $\pi^{-1}_{eq}(y_0)\subseteq W$.
Thus there exists  open neighborhood $V$ of $y_0$ such that $\pi_{eq}^{-1}(V)\subseteq  W$.
Since $(X_{eq}, S)$ is equicontinuous, we can take a compatible metric $d_{eq}$ on $X_{eq}$ such that
$$d_{eq}(Sy_1,Sy_2)=d_{eq}(y_1,y_2)$$
for any $y_1,y_2\in X_{eq}$.

We choose  $\varepsilon>0$ such that  $\{ y\in X_{eq}:d_{eq}(y,y_0)<2 \varepsilon\}\subseteq V$. Let $$V_1:=\{ y\in X_{eq}:d_{eq}(y,y_0)<\varepsilon\}.$$
Let $$M=N_{S}(y_0, V_1):=\{n\in \mathbb{N}: S^n y_0\in V_1\}.$$
Then $M$ is  syndetic and for any $n\in M$, $S^n {y_0}\in V_1$ then implies $S^nV_1\subseteq V$.
Since $N_T(\pi_{eq}^{-1}(V_1),\delta;r)$ is thick, we can take $k\in M \cap N_T(\pi_{eq}^{-1}(V_1),\delta;r)$.

Then on one hand there exist $z_1, \cdots, z_r\in \pi_{eq}^{-1}(V_1)$
such that $d(T^kz_i,T^kz_j)>\delta$ for any $1\le i\neq j\le r$.
On the other hand as $k\in M$,
$\pi_{eq}(T^kz_i)=S^k\pi_{eq}(z_i)\in S^k(V_1)\subseteq V$
for $i=1,\cdots,r$. Thus $\{T^kz_1,\cdots,T^kx_r\}\subseteq W$.
Note that $r>m$, we can find $1\le a\neq b\le r$ and $i\in \{1,2,\cdots,m\}$ such that
$d(T^kz_a,x_i)<\frac{\delta}{3}$ and $d(T^kz_b,x_i)<\frac{\delta}{3}$.
This implies $d(T^kz_a,T^kz_b)<\frac{2\delta}{3}$, a contradiction.
This show that  $\eta_{r}>0$.

Conversely assume $\eta_r>0$ and let $\delta _0=\frac{\eta _{r}}{2}$, now we will show that $(X,T)$ is thickly $r$-sensitive with a sensitive constant $\delta=\frac{\delta_0}{4}$.
Actually we are to show that $(X,T)$ is thickly r-sensitive with the same sensitive constant $\delta$ for $\mu$, where $\mu$ is any given $T$-invariant Borel probability measure on $X$.
It is sufficient to show that for any set $A \in \mathcal{B}_{X,\mu}^+ $ and $L\in \mathbb{N}$, we can find $m\in \mathbb{Z}_+$ such that
$$\{ m,\cdots,m+L-1\} \subset N_T(A,\delta;r).$$

We fix  a point $y\in X_{eq}$. Since $\eta_{r}>0$, for each $0\le k\le L-1$ we can find $\{x_1^k,\cdots,x^k_r \}\subset{ \pi_{eq}^{-1}(S^ky)}$ such that $d(x^k_i,x^k_j)\geq \delta_0$ for any $1\leq i<j\leq r $.

For $0\leq k\leq L-1$ and $ 1\leq i\leq r$, let $$ W_i^k:= \{ x\in X:  d(x,x_i^k)<\delta \} $$
and
$$ z_i^k=T^{-k}x_i^k\text{ and }  U_i^k=T^{-k}W_i^k.$$
Then $z_i^k\in \pi_{eq}^{-1}(y)$ and $U_i^k$ is an open neighborhood of $z_i^k$. It is also clear that $z_i^k\neq z_j^k$ for any $0\leq k\leq L-1$ and $1\leq i< j\leq r$.

Note that $\{z_i^k:0\le k\le L-1,1\le i\le r\}\subseteq \pi^{-1}_{eq}(y)$. Hence $(z_i^p, z_j^q)\in Q(X,T)$ for any $0\leq p,q\leq L-1$ and $ 1\leq i, j\leq r$.
Thus since $(X,T)$ is minimal,
$$(z_i^k)_{0\le k\le L-1,1\le i\le r}\in Q_{Lr}(X,T)\setminus \triangle_{Lr} (X)=S^\mu_{Lr}(X,T)$$
by Proposition \ref{xyz}. Then we can find $ m\in \mathbb{Z}_+ $ such that $ T^m A\cap U_{i}^k\neq \emptyset$ by the definition of $S^\mu_{Lr}(X,T)$.

This implies we can find $\omega _i^k\in A $ such that $T^m\omega _i^k \in U_i^k$ and so $T^{m+k}\omega _i^k \in W_i^k$ for any  $0\leq k\leq L-1$ and $ 1\leq i\leq r$.  We know
$$d(T^{m+k}\omega _i^k,T^{m+k}\omega _j^k)\ge dist(W_i^k,W_j^k) >\delta$$
for any $0\leq k\leq L-1$ and $1\leq i< j\leq r $ from the construction of $W_i^k$.

Thus $\{m, m+1,\cdots, m+L-1 \}\subset  N_T(A,\delta;r)$.  So $(X,T)$ is thickly r-sensitive with the sensitive constant $\delta$ for $\mu$.
Since $(X,T)$ is minimal, every non-empty open subset of $X$ belongs to $\mathcal{B}_{X,\mu}^+$. Thus $(X,T)$ is also thickly $r$-sensitive.
This finish the proof of Proposition \ref{pro-2}.
\end{proof}

\begin{proof}[Proof of Theorem \ref{thm-1}] As a direct corollary of Proposition \ref{pro-1} and Proposition \ref{pro-2} we can get Theorem \ref{thm-1}.
\end{proof}

\section{Proof of Theorem \ref{thm-2}}
In this section we are to prove Theorem \ref{thm-2}. For that we need some notation and Propositions. For a TDS $(X,T)$, $\delta>0$ and $r\in \mathbb{N}$ with $r\geq 2$, and a non-empty subset $U$ of $X$, recall that $ N_T(U,\delta;r)$ is defined in Introduction.
We can also describe $ N_T(U,\delta;r)=\{n\in \mathbb{N}:diam_r (T^nU)>\delta \}$, where the $r$-version diameter $diam_r(\cdot)$ is defined as follows: for any non-empty subset $ B\subset X$,
\begin{equation}\label{hd0}
diam_r(B)=\sup\{\min \limits_{1\leq i\neq j\leq r} d(x_i,x_j): x_1,x_2,\cdots,x_r \in B \}.
\end{equation}

For a TDS $(X,T)$, let $2^X$ be the set of all non-empty closed subsets of $X$. Recall Hausdorff metric $H_d$ on $2^X$ was defined as :
\begin{equation}\label{hd}
H_d(A,B)=\max\{\max_{x\in A}d(x,B),\max_{y\in B}d(A,y)\}
\end{equation}
for any $A,B\in 2^X$.

\begin{prop} \label{Con-d} Let $(X,T)$ be a TDS and $r\in \mathbb{N}$ with $r\geq 2$, $diam_r(\cdot)$ and $H_d$ are defined as above \eqref{hd0} and \eqref{hd}. Then $diam_r(\cdot)$ is a continuous function on $(2^X,H_d)$, (i.e., if $\lim_{n\rightarrow \infty} A_n= B$ with respect to the Hausdorff metric $H_d$, then $diam_r(B)=\lim_{n\rightarrow\infty} diam_r(A_n)$).
\end{prop}
\begin{proof} Let $\{A_n\}_{n=1}^\infty\subset 2^X$ and $B\in 2^X$ such that $\lim_{n\rightarrow \infty} A_n= B$ with respect to the Hausdorff metric $H_d$.
For simplicity we write $diam_r(B)=R$. Then for any $\epsilon >0$ there exists $N=N(\epsilon) \in \mathbb{N}$ such that $H_d(A_n,B)<\epsilon$ when $n\ge N$.
When $n\ge N$ for any $a^n_1,\cdots,a^n_r \in A_n$, there are $b^n_1,\cdots,b_r^n \in B$ such that $\max_{1\le i\le r}d(a^n_i,b^n_i)<\epsilon$  according to the definition of $H_d$. So we have
\begin{align*}
\min_{1\le i\neq j\le r} d(a^n_i,a^n_j)&\leq \min_{1\le i\neq j\le r} \left( d(a^n_i,b^n_i)+d(b^n_i,b^n_j)+d(b^n_j,a^n_j)\right)\\
&<2\epsilon+\min_{1\le i\neq j\le r} d(b^n_i,b^n_j).
\end{align*}
This implies $diam_r(A_n)<2\epsilon +R$ when $n\ge N$ according to the definition of $diam_r(\cdot)$. So $\limsup_{n\rightarrow\infty} diam_r(A_n)\leq 2\varepsilon +R $. Let $\epsilon\searrow 0$ we have $\limsup_{n\rightarrow\infty} diam_r(A_n)\leq R$.

Conversely, suppose $\liminf_{n\rightarrow\infty}diam_r(A_n)=R'$. For any $\epsilon>0$ there exists $n\in \mathbb{N}$ such that $diam_r(A_n)<R'+\epsilon$ and $H_d(A_n,B)<\epsilon$. Similar to the above analysis we know $diam_r(B)\leq R'+2\varepsilon$, that implies $diam_r(B)\leq R'$ when $\varepsilon\rightarrow0$. So we get $diam_r(B)=\lim_{n\rightarrow\infty} diam_r(A_n)$.
This finishes the proof of Proposition \ref{Con-d}.
\end{proof}

Resembling the discussion in Proposition 3.5 of \cite{Y}, we have the following Proposition which is the measure-theoretic corresponding of \cite[Theorem 3.4]{Z}.

\begin{prop}\label{pro-3} Let $(X,T)$ be a TDS with a $T$-invariant Borel probability measure $\mu$. The following statements are equivalent:
\begin{enumerate}
  \item[(1)] $(X,T)$ is multi-$r$-sensitive for $\mu$.
  \item[(2)] $(X,T)$ is thickly $r$-sensitive for $\mu$.
  \item[(3)] $(X,T)$ is thickly syndetically  $r$-sensitive for $\mu$.
\end{enumerate}
\end{prop}
\begin{proof} (1)$\Rightarrow $(2) Suppose $(X,T)$ is multi-$r$-sensitive for $\mu$ with a sensitive constant $\delta>0$. For any $A\in \mathcal{B}_{X,\mu}^+ $ and any $L\in \mathbb{N}$, one has $T^{i}A\in \mathcal{B}_{X,\mu}^+ $ for $i=0,1,\cdots,L$ since $\mu$ is $T$-invariant. From the definition of multi-$r$-sensitive for $\mu$, there exists $n_L\in \mathbb{N}$ such that $n_L\in \bigcap_{i=0}^L N_T(T^{i}A,\delta; r)$. Thus $\{n_L,n_L+1,\cdots,n_L+L\}\subset N_T(A,\delta; r)$, which implies $(X,T)$ is thickly $r$-sensitive for $\mu$.

\medskip
(2)$\Rightarrow $(3) Suppose $(X,T)$ is thickly $r$-sensitive for $\mu$ with a sensitive constant $\delta>0$, then we claim $(X,T)$ is thickly syndetically $r$-sensitive for $\mu$ with the same sensitive constant. If not, then there exists $A\in \mathcal{B}_{X,\mu}^+$  such that
$N_T(A,\delta;r)$ is not thickly syndetic. Without loss of generality, we may assume that $A$ is a closed subset of $X$.

Similar to the analysis in the proof of Proposition 3.5 (2)$\Rightarrow $(3) in \cite{Y}, we can find $p\in \mathbb{N}$,
$$\bigcup_{d=1}^\infty \{ n_d^1,n_d^2,\cdots,n_d^d \}\subset \mathbb{N}\setminus N_T(A,\delta;r), $$
$\{ a_i\}_{i\in \mathbb{N}}\subset \mathbb{Z}_+$ and $B\in 2^X$
 such that $a_1=0$, $1\le n_d^{i+1}-n_d^{i}=a_{i+1}-a_{i}\leq p$ for any $d,i \in \mathbb{N}$ with $d\ge i+1$,
 and $\lim_{d\rightarrow \infty} T^{n_d^i}A= T^{a_i} B$ with respect to the Hausdorff metric $H_d$ for each $i\in \mathbb{N}$.
It's clear that
$$\mu(B)\ge \limsup_{d\rightarrow \infty} \mu(T^{n_d^1}A)=\mu(A)>0$$ and so $B\in \mathcal{B}_{X,\mu}^+ $.

Now on one hand, $N_T(B,\delta; r)$ is  thick since $(X,T)$ is thickly $r$-sensitive for $\mu$ with a sensitive constant $\delta>0$. On the other hand, by Proposition \ref{Con-d}
$$diam_r(T^{a_i}B)=\lim_{d\rightarrow \infty} diam_r(T^{n_d^i}A)\leq \delta$$
for any $i\in \mathbb{N}$. This implies  $\{ a_i\}_{i\in \mathbb{N}} \cap N_T(B,\delta; r)= \emptyset$.
Hence $N_T(B,\delta; r)$ is not thick since $\{a_i\}_{i=2}^\infty$ is syndetic. This is a contradiction.

\medskip
(3)$\Rightarrow $(1) It's abvious because the intersection of finitely many thickly syndetic sets is still a thickly syndetic set.
\end{proof}

\begin{proof}[Proof of Theorem \ref{thm-2}] We have proved the equivalence of 1), 2) and 3) in Proposition \ref{pro-3}. Suppose 4) is established, we have $\eta_{r}>0$ according to Proposition \ref{pro-1}, so $(X,T)$ is thickly $r$-sensitive for $\mu$ according to the proof of Proposition \ref{pro-2}. On the other hand, every non-empty open set $U$ belongs to $\mathcal{B}_{X,\mu}^+$ since $(X,T)$ is minimal. Thus if $(X,T)$ is thickly $r$-sensitive for $\mu$ then $(X,T)$ is thickly $r$-sensitive. This leads to 4) by Theorem \ref{thm-1} (see also Proposition \ref{pro-1} and Proposition \ref{pro-2}). Therefore the equivalence of those conditions are proved.
\end{proof}

\bibliographystyle{amsplain}

\end{document}